\documentclass{amsart}
\usepackage{amsmath,amssymb,amsthm}
\usepackage[dvips]{color}
\newcommand{\R}{{\mathbb R}}

\newcommand{\C}{{\mathbb C}}
\newcommand{\Z}{{\mathbb Z}}

\newcommand{\bz}{\boldsymbol{z}}

\newcommand{\bm}{\boldsymbol{m}}

\newcommand{\PD}[1]{\frac{\partial}{\partial {#1}}}
\newcommand{\OL}[1]{\overline{#1}}
\newcommand{\WT}[1]{\widetilde{#1}}

\newcommand{\bS}{\partial\Sigma}
\newcommand{\Si}{\Sigma}

\newcommand{\I}{\sqrt{-1}}

\newcommand{\CL}{\mathcal{L}}

\newcommand{\bsg}{\boldsymbol{\Sigma}}

\newtheorem{theorem}{Theorem}[section]
\newtheorem{corollary}[theorem]{Corollary}
\newtheorem{lemma}[theorem]{Lemma}
\newtheorem{proposition}[theorem]{Proposition}
\newtheorem{proposition-definition}[theorem]{Proposition - Definition}
\newtheorem{definition}[theorem]{Definition}
\newtheorem{remark}[theorem]{Remark}
\newtheorem{example}[theorem]{Example}

\title[Chern-Weil Maslov index and its orbifold analogue] {Chern-Weil Maslov index\\ and its orbifold analogue}
\author{Cheol-Hyun Cho, Hyung-Seok Shin}

\begin{document}
\thanks{Authors were supported by the Basic research fund 2009-0074356 funded by the Korean
government.}
\begin{abstract}
We give  Chern-Weil definitions of the Maslov indices of bundle pairs over a Riemann surface $\Sigma$ with boundary, which consists of symplectic vector bundle on $\Sigma$ and a Lagrangian subbundle on $\partial \Sigma$ as well as its generalization for transversely intersecting Lagrangian boundary
conditions.
We discuss their properties and relations to the known topological definitions. As a main application, we extend Maslov index to the case with orbifold interior singularites, via curvature integral, and find also an analogous topological definition in these cases.
\end{abstract}

\maketitle

\section{Introduction}
Maslov index, which plays  important roles in several contexts of geometry and analysis, has been extensively studied in the literature such as in \cite{Ar},\cite{GS},\cite{M},\cite{SZ} to name a few.
For an excellent review, we refer readers to Cappell, Lee and Miller \cite{CLM}.

The Maslov index, which we consider  in this paper, is associated to a bundle pair over a Riemann surface with boundary( or with boundary and boundary punctures). By bundle pair  $(E,L)$ over $\Sigma$, we  mean a symplectic vector bundle $E \to \Sigma$ equipped with compatible   complex structure, together with Lagrangian subbundle
$L \to \partial \Sigma$ over the boundary of $\Sigma$.
In this paper, we give another definition of the Maslov index via curvature integral of orthogonal connections.  This is analogous to the case of $c_1$, which can be defined via Chern-Weil theory as
a curvature integral of connections.

As we work with maps from  Riemann surfaces with boundary and Lagrangian boundary condition on them, we require
the certain orthogonal condition(to preserve the Lagrangian bundle data) of the connection on the boundary $\partial \Sigma$ to define the index.
This is somewhat opposite to the topological definition of the Maslov index. Namely, Maslov index can be defined by taking a trivialization of $E$ over $\Sigma$ and measure how the Lagrangian subbundle data is twisted along the boundary. But in Chern-Weil definition, we require the orthogonal condition on the boundary and measure how much the connection is twisted over the interior of $\Sigma$ by taking an integral of its curvature. We provide two proofs of the equality of the proposed Chern-Weil Maslov index and the standard Maslov index which is  defined
as winding number in Lagrangian grassmanian.

We also consider its generalization, arising from the maps with boundary on transversely intersecting Lagrangian submanifolds, as in the definition of Fukaya category(see \cite{Fu} for example). Chern-Weil definition of the Maslov index also works in this case, and we will find a relation of such an index with the Fredholm index of the related Cauchy-Riemann operator.

The main motivation for us to develop the Chern-Weil definition of Maslov index is to extend the Maslov index to the  orbifold setting.
In the last section, we discuss the case with interior orbifold singularites. Namely, we will consider
bundle pairs over Riemann surface $\bsg$ with boundaries and interior orbifold singularities. In this case, we consider not vector bundles but orbi-bundles, and the standard topological definition is not available since
 orbi-bundles are {\em not} trivial bundles over $\Sigma$. But the Chern-Weil definition which we give in this paper, can be easily extended to this setting.

 Alternative approach to define Maslov index in this orbifold case is to take branch covering
 $br:\WT{\Si} \to \bsg$ by a smooth Riemann surface $\WT{\Si}$ with boundary, and
 define the Maslov index to be that of the pull-back bundle over $\WT{\Si}$ divided by
 the degree of the branch covering map $br$. We use Chern-Weil definition to prove that
such an index is independent of the choice of a branch covering map, and prove that it is the same as Chern-Weil Maslov index.

 We also discuss the relationship with the desingularized Maslov index which is defined by the first author and Poddar in \cite{CP} (following the desingularization of Chen and Ruan \cite{CR})

Here is the outline of the paper. In section 2, we give a Chern-Weil definition of the Maslov index of
a bundle pair, and in section 3, we gives two proofs of the theorem that Chern-Weil index equals the usual topological Maslov index.  The first proof is easier, but the second proof extends to the case of orbifolds.
In section 4, we show properties of Chern-Weil index such that the definition does not depend on the choice of orthogonal connection or an   complex structure.
In section 5, we consider the case of transversely intersecting Lagrangian boundary condition and compare it to the Fredholm index of related Cauchy-Riemann operator. In section 6, we extend the above definition to the orbifold case, establish a topological definition using branch covering maps, and find a relation to the desingularized Maslov index.

\section{Maslov index via orthogonal connection}
In this section, we define an $L$-orthogonal connection( c.f. \cite{V}) of a bundle pair to give a Chern-Weil definition of its Maslov index.

We recall the well-known definition of $c_1(E)$ of a complex line bundle $E$ via curvature integral.
\subsection{Chern-Weil definition of the first chern class}
Let $\nabla$ be a connection of a complex line bundle $E$ over a closed surface $\Sigma$, and denote by $F_\nabla$ its curvature.
 The following theorem is well-known
\begin{theorem}
The curvature $F_\nabla$ satisfies the following:
\begin{enumerate}
\item $dF_\nabla=0$
\item If $\nabla$ and $\nabla'$ are two connections of $E$, then
$\nabla = \nabla' + \eta$ for a 1-form $\eta$ and
$F_\nabla = F_{\nabla'} + d\eta$
\item  The first chern number $c_1(E)([\Si])$ is given by
\begin{equation}\label{chernintegral}
c_1(E)([\Si]) = \frac{ \I}{2 \pi}\int_{\Si}F_\nabla,
\end{equation}
and it is independent of the choice of a connection.
\end{enumerate}
\end{theorem}
In fact, $\eta$ is a $End(E)$-valued 1-form. However, since $End(E)$ is a trivial line bundle, we can consider the $\eta$ as a 1-form on $\Si$ after fixing a trivialization of $End(E)$.
Note that the difference of curvature integrals for two connections $\nabla$ and $\nabla'$
is
$$\int_{\Si}F_\nabla - \int_{\Si}F_{\nabla'} = \int_{\Si} d\eta = \int_{\partial \Si} \eta =0.$$
But for the case with Riemann surfaces with boundary, and  the invariance of
the curvature integrals does not hold for arbitrary connections since $\partial \Si \neq 0$.

To obtain an invariant curvature integrals for the case with boundaries, we  introduce the notion of an orthogonal connection.
\subsection{Orthogonal connection for a bundle pair}
We first recall a definition of bundle pair. Let $\Sigma$ be a Riemann surface with boundary $\bS$.
\begin{definition}
We denote by a {\em bundle pair $(E, L) \to (\Si,\bS)$}, a symplectic  vector bundle $(E,\omega_E)$ over $\Si$, a Lagrangian subbundle   $L$ over $\partial \Si$.
\end{definition}
We also consider a compatible complex structures $J$ of $(E,L)$ which makes $E$ a complex
vector bundle with an induced inner product $g(\cdot,\cdot) = \omega(\cdot,J\cdot)$.
Denote by $g_\C = g + \sqrt{-1} \omega$  the induced hermitian inner product of $E$.
A connection $\nabla$ is said to be {\em unitary} if $\nabla$ is compatible
with the metric $g_\C$, or equivalently, the holonomy of $\nabla$ lies in $U(n)$.

\begin{definition}
Let $\nabla$ be a unitary connection on $E \to \Si$.
Then, $\nabla$  is called {\em $L$-orthogonal } if the parallel transport  along $\bS$ via $\nabla$ preserves Lagrangian subbundle $L \to \bS$.
\end{definition}

To construct such an $L$-orthogonal connection of a bundle pair $(E,L) \to (\Si,\bS)$,
we can proceed as follows.
Given a $g$-orthogonal  metric connection $\nabla$ for $L$,
by defining $\nabla J e = J \nabla e$  for any local section $e$ of $L$, we can
extend the connection to $E \to \bS$, and by trivially extending to the neighborhood of $\bS$ and using the partition of unity, we can extend it to a unitary connection to $E \to \Si$.

\begin{remark}
The name, orthogonal connection is given following the work of Vaisman, who considered orthogonal unitary connection in \cite{V}. He considered  principal $Sp(2n)$-bundles  and
related principle $U(n)$ subbundle (by choosing an   complex structure), and
a  unitary connection  preserving certain principal $O(n)$-subbundle ( which is defined from a  real unitary frame of Lagrangian subbundle) was called orthogonal unitary connection. He considered the case that the Lagrangian subbundle is defined everywhere(not just on the boundary) and hence, the Maslov index vanishes in these cases. He used orthogonal unitary connections and Chern-Weil theory to
study secondary invariants.
\end{remark}

Note that we only require orthogonality along $\bS$. Another natural (and more restrictive)
assumption would be to take a tubular neighborhood of $\bS$ and require the connection to be a product form, i.e. it is a pullback of the orthogonal connection on $\bS$ along normal direction.
But we remark that  the resulting curvature integral will be the same, which can be proved as in the proof  of the proposition \ref{lem:stoke}

\begin{example}\label{ex1}
Consider a Lagrangian submanifold $L=S^1 \subset \C$, where $\C$ is equipped with standard symplectic structure $\omega_0=dx \wedge dy$. Consider  inclusion of a unit disc $u: D^2 \subset \C$
(so that $u(\partial D^2) = L$).

Consider the pull-back bundle $u^*T\C \cong D^2 \times \C$  and its Lagrangian subbundle
$u|_{\partial D^2}^*TL$.
Consider the trivialization of $u^*T\C$ as above and denote by $(r,\theta)$ the polar coordinate of $D^2$.
For the complex frame $\{\epsilon(r,\theta):= \left. \PD{x} \right|_{(r,\theta)}\}$ of $u^*T\C$ ,
we define a connection $$\nabla:=d-\sqrt{-1}rd\theta,$$

In this trivialization, $\sqrt{-1}e^{\sqrt{-1}\theta}\epsilon(1,\theta)$ defines a real frame of $L$ and
one can check that
 $$\nabla_{\PD{\theta}} \sqrt{-1}e^{\sqrt{-1}\theta}\epsilon(1,\theta)=
 \big( (d - \sqrt{-1}rd\theta)\sqrt{-1}e^{\sqrt{-1}\theta} \big)(\PD{\theta})\epsilon(1,\theta)=0  .$$
 Therefore the connection $\nabla$ is an $L$-orthogonal unitary connection.
\end{example}

\subsection{Maslov index}
Recall the definition of a Maslov index for a bundle pair $(E, L) \to \Si$, where
$E$ is a symplectic vector bundle over $\Si$,  and $L$ is a Lagrangian subbundle over $\partial \Si$.
Let $J$ be a compatible   complex structure on $E$, and consider $E$ as a complex vector bundle.

Recall the following well-known lemma.
\begin{lemma}[\cite{Oh}]
Consider the subset
$$\tilde{\mathcal{U}}(n)=\{A\in U(n,\C)|A=A^t\}.$$
 Then the map
$$B:\Lambda(n)=U(n)/O(n)\to \tilde{\mathcal{U}}(n); A \mapsto A\bar{A}^{-1}$$
is a diffeomorphism.
\end{lemma}
The  Maslov index $\mu(\gamma)$ of an oriented loop $\gamma:S^1\to \Lambda(n) $ is defined to be the winding number of
$$\det\circ B\circ\gamma:S^1\to\C\backslash\{0\}.$$

Now given a  bundle pair $(E, L)\to\Si$, if  $\partial\Si\neq\varnothing$, then  vector bundle $E\to\Si$ can be trivialized.  We fix a symplectic trivialization $\Phi:E\cong\Si\times\C^{n}$, and let $R_1,...,R_h$ be the connected components of $\partial\Si$, with orientation induced by the orientation $\Si$. Then $\Phi(L|_{R_i})$ gives a loop $\gamma_i:S^1\to \Lambda(n)$. Let us denote $\mu(\Phi,R_i):=\mu(\gamma_i)$.

\begin{definition}\label{def:mas}
The Maslov index of the bundle pair $(E,L)$ is defined by
$$\mu(E,L)=\sum_{i=1}^h \mu(\Phi,R_i)$$
where $\Phi:E\to\Si\times\C^n$ is any trivialization.
\end{definition}
The Maslov index is independent of the choice of trivialization $\Phi$, and the choice of an   complex structure $J$. (see \cite{KL} for example.)
\subsection{Chern-Weil Maslov index}
The main objective of this paper is to  give another definition of the Maslov index $\mu_{CW}$ for
the bundle pair $(E,L)\to\Si$ in terms of curvature integral:
\begin{definition}
Let $\nabla$ be a connection on $E$ which restricts, on the boundary of $\Si$, to an $L$-orthogonal unitary connection on $(E|_{\partial\Si},J)$.
The Maslov index of the bundle pair $(E,L)$ is defined by
$$\mu_{CW}(E,L)=\frac{ \I}{\pi}\int_\Si{tr(F_\nabla)}$$
where $F_\nabla\in\Omega^2(\Si,End(E))$ is the curvature induced by $\nabla$.
\end{definition}
\begin{remark}
Note that the denominator of \eqref{chernintegral} is $2\pi$.
\end{remark}

We consider the example \ref{ex1}.
\begin{example}
For the connection $\nabla$ defined in example \ref{ex1},
we have
$$F_{\nabla}=d(-\sqrt{-1}rd\theta)=-\sqrt{-1}dr\wedge d\theta$$
Hence,
 $$\frac{ \I}{\pi}\int_{D^2}{tr(F_\nabla)} = 2.$$
This shows that $\mu_{CW}=2$ and it is equal to the topological Maslov index.
\end{example}

In the following section, we prove that $\mu_{CW}(E,L)$ is independent of the choice of
the orthogonal connection and equal the topological Maslov index $\mu(E,L)$.

\section{Equivalence of two Maslov indices}
We will give two proofs of equivalence of two Maslov indices.
\begin{theorem}\label{thm:main}
Given a bundle pair $(E,L)$, topological Maslov index equals Chern-Weil Maslov index:
$$\mu(E,L) = \mu_{CW}(E,L).$$
\end{theorem}
The first proof in subsection 3.1 is easier, but the second proof in subsection 3.2 using doubling construction can be extended to the case of orbifolds, and will be used in a later section.

\subsection{First proof of $\mu = \mu_{CW}$}
\begin{proof}
Consider a bundle pair $(E,L)$ with orthogonal connection $\nabla$.
We fix an   complex structure $J$ of $E$ and regard $E$ as a complex vector
bundle.
Consider $\Lambda^{n}E$ the top exterior bundle of $E$, with an induced connection $\WT\nabla$.
We have a trivialization
$\Phi:E\to\Si\times\C^n$ as a complex vector bundle since $\partial \Sigma \neq 0$.
With respect to the standard frame $\{\epsilon_1,\cdots,\epsilon_n\}$ of $\Si \times \C^n$,
we can write $\nabla=d+A$ for some $n\times n$-matrix-valued one form $A=(a_{ij})$. Then
it is easy to see that $\WT\nabla=d+tr(A)$ with respect to the frame$\{\epsilon:=\epsilon_1\wedge\cdots\wedge\epsilon_n\}$.

Recall that the curvature of $\nabla$ and $\WT\nabla$ is given as
$$F_\nabla = dA + A\wedge A,\; F_{\WT\nabla}=d(tr(A)).$$

\begin{lemma}
$$\int_\Si tr(F_\nabla)=\int_\Si tr(F_{\WT\nabla}).$$
\end{lemma}
\begin{proof}
$$tr(A\wedge A)=\sum_{i,j}a_{ij}\wedge a_{ji}=\sum_i a_{ii}\wedge a_{ii}=0.$$
Second equality follows from cancelation of $a_{ij}\wedge a_{ji}$ with $a_{ji}\wedge a_{ij}=-a_{ij}\wedge a_{ji}$ for $i\neq j$.
\end{proof}

Now, we recall the standard relation between holonomy and curvature integral for line bundles.
Given a complex line bundle $\CL$ over a manifold $M$ and a connection $\nabla'$ on $\CL$,
the holonomy along a contractible loop $\gamma$ (which is bounded by the 2-dimensional contractible domain $D \subset M$) is given by
$$Hol_\gamma(\CL,\nabla') = exp( -\int_D F_{\nabla'}).$$

Note that if $\nabla'=d+A $ is a unitary connection on $D$, then $A$ satisfies $A =- \overline{A}^t$.
Therefore for a complex line bundle, $A$ is the purely imaginary connection 1-form.
 Since
$\xi(t)=e^{-\int_0^t A(\dot{\gamma}(s))ds}\xi(0)$ is a parallel transformation of $\xi(0)$,
the integral $- \frac{1}{\sqrt{-1}}\int_D F_{\nabla'} = \sqrt{-1}\int_D F_{\nabla'}$ gives the rotation angle of a parallel section along $\gamma$.
 Note that $\int_\gamma A=\int_D F_{\nabla'}$ by Stokes' theorem in this case.

In general, the above relation for $D$ extends to the case of Riemann surface $\Sigma$ with boundary $\partial \Sigma$. Namely, the integral $\sqrt{-1}\int_\Si F_{\nabla'}$ gives the sum of rotation angles of parallel sections along boundaries $\bS$ with the induced orientations from $\Si$.

Now we apply this to  $\Lambda^{n}E$ and $\WT\nabla$.
Note that, by the definition of $L$-orthogonal unitary connection, parallel transformation in $E$ preserves frame vectors of $L$. More precisely,  if $\{e_1(t),\cdots,e_n(t)\}$ is a horizontal sections of $E$ along $\gamma(t) \subset \bS$ and $\{e_1(t_0),\cdots,e_n(t_0)\}$ is a real orthogonal frame of $L_{t_0}$ at some moment $t_0$, then $\{e_1(t),\cdots,e_n(t)\}$ would be a real orthogonal frame of $L_t$ for all $t$.

Define a matrix $u(t):=(e_1(t),\cdots,e_n(t))\in U(n)$ using the frame as column vectors.
Then, we have (for the standard frame $\{\epsilon_1,\cdots,\epsilon_n\}$)
 $$e_1(t)\wedge\cdots\wedge e_n(t)=det(u(t))\epsilon_1\wedge\cdots\wedge\epsilon_n.$$

  So $det(u(t))$ is a  frame of the Lagrangian subbundle $\Lambda^n L\subset \Lambda^n E$.
   In the trivialization $det(\Phi):\Lambda^n E\to\Si\times\C$, we have $det(u(t)) \in U(1)$.

Observe that the $det(u(t))$ gives a horizontal section.
Hence the $\sqrt{-1}\int_\Si F_{\WT\nabla}$ measures the rotating angle of $det(u(t))$ in $U(1)$.

 As topological Maslov index $\mu$ corresponds to the rotation number of $det^2(u(t))$ in $U(1)$,  hence it is equal
 to $$\frac{\sqrt{-1}}{\pi}\int_\Si F_{\WT\nabla}.$$
\end{proof}

\subsection{Second proof of $\mu = \mu_{CW}$}
 We  will use the doubling construction (and the equivalence between the topological and Chern-Weil definition of the first Chern class).

To explain the doubling construction, we
recall the following well-known theorem (see \cite{AG}).
\begin{theorem}\label{doublesurface}
    Let $\Sigma$ be a bordered Riemann surface. There exists a double cover
    $\pi:  \Sigma_\C\rightarrow \Sigma$ of $\Sigma$ by a compact Riemann surface
    $\Sigma_\C$ and an antiholomorphic involution $\sigma:\Sigma_\C\rightarrow\Sigma_\C$
    such that $\pi\circ \sigma=\pi$. There is a holomorphic embedding
    $i: \Sigma\rightarrow \Sigma_\C$ such that $\pi\circ i$ is the identity map.
    The triple $(\Sigma_\C,\pi,\sigma)$ is unique up to isomorphism.
   \end{theorem}

    \begin{definition}
    We call the triple $(\Sigma_\C,\pi,\sigma)$ in  Theorem \ref{doublesurface} the
    \emph{complex double} of $\Sigma$, and $\OL{\Sigma}=\sigma(i(\Sigma))$
    the \emph{complex conjugate} of $\Sigma$.
   \end{definition}

 We recall the following theorem 3.3.8 of \cite{KL} and its proof for reader's convenience.

 \begin{theorem}
  \label{degree}
   Let $(E,L)$ be a bundle pair over a bordered Riemann surface
   $\Sigma$.
   Then there is a complex vector bundle $E_\C$ on $\Sigma_\C$ together
   with a conjugate linear involution $\tilde{\sigma}: E_\C\rightarrow E_\C$
   covering the antiholomorphic involution $\sigma: \Sigma_\C\rightarrow\Sigma_\C$
   such that $E_\C|_\Sigma=E$ (where $\Sigma$ is identified with its image
   under $i$ in $\Sigma_\C$) and the fixed locus of $\tilde{\sigma}$
   is $L \rightarrow \bS$. Moreover, we have
\[
   \deg E_\C=\mu(E,L).
\]
  \end{theorem}

\begin{proof}
   Let $R_1,\ldots,R_h$ be the connected components
   of $\bS$, and let $N_i\cong R_i\times [0,1)$ be a neighborhood of
   $R_i$ in $\Si$ such that $N_1,\ldots,N_h$ are disjoint. Then
   $(N_i)_\C=N_i\cup \overline{N}_i$ is a tubular neighborhood of $R_i$ in
   $\Si_\C$, and $N\equiv\cup_{i=1}^h (N_i)_\C$ is a tubular
   neighborhood of $\bS$ in $\Si_\C$. Let $U_1=\Si\cup N$,
   $U_2=\OL{\Si}\cup N$, so that $U_1\cup U_2=\Si_\C$ and
   $U_1\cap U_2=N$.

   Fix a trivialization $\Phi: E\cong\Si\times\C^n$, where $n$ is
   the rank of E. Then $\Phi(L|_{R_i})$ gives rise to a loop
   $B_i: R_i\rightarrow \widetilde{\mathcal{R}}_n\subset GL(n,\C)$.
   To construct $E_\C\rightarrow \Si_\C$, we glue trivial
   bundles $U_1\times \C^n\rightarrow U_1$ and
   $U_2\times \C^n\rightarrow U_2$ along $N$ by identifying
   $(x,u)\in (N_i)_\C\times \C^n\subset U_1\times \C^n$
   with $(x,B_i^{-1}\circ p_i(x)u)\in(N_i)_\C\times
   \C^n\subset U_2\times \C^n$,
   where $p_i: (N_i)_\C\cong R_i\times(-1,1)\rightarrow R_i$ is the
   projection to the first factor and
   $B_i^{-1}: R_i\to\widetilde{\mathcal{R}}_n$ denotes the map $B_i^{-1}(x)=
   (B_i(x))^{-1}$. There is a conjugate linear
   involution $\tilde{\sigma}: E_\C\rightarrow E_\C$ given by
   $(x,u)\in U_1\times\C^n\mapsto (\sigma(x),\bar{u})\in U_2\times\C^n$
   and
   $(y,v)\in U_2\times\C^n\mapsto (\sigma(y),\bar{v})\in U_1\times\C^n$.
   It is clear from the above construction that
    $\tilde{\sigma}: E_\C\rightarrow E_\C$ covers the antiholomorphic
   involution $\sigma: \Si_\C\rightarrow\Si_\C$, and the fixed locus of
   $\tilde{\sigma}$ is $L\rightarrow \bS$.

 Proof of \cite{MS} Theorem 2.69 shows that $\deg(E_\C)$, which is
 $(c_1(E_\C) \cap \Sigma_\C)$ can be defined by winding number (degree) of  the overlap map from trivializations of $E_\C$ over $\OL{\Sigma}$ to that of $\Sigma$. But in our setting, this map is given by $B_i$ whose winding number defines the Maslov index $\mu(E,L)$.
This proves the desired identity.

\end{proof}

Now, given an orthogonal connection $\nabla$ on a bundle pair $(E,L)$, let us assume that it has a product form near the boundary.
More precisely, on the normal neighborhood $N:=\bS\times[0,1)$, we have $\nabla|_{\bS\times[0,\epsilon)}=\pi^*(\nabla|_{\bS})$ where $\pi:\bS\times[0,\epsilon)\to\bS$ is the projection map.
It is easy to see that such an orthogonal connection always exists.

We construct
a connection $\nabla_\C$ on the complex double $E_\C$ from the orthogonal connection $\nabla$ on a bundle pair $(E,L)$.
Fix a trivialization $\Phi: E\cong\Si\times\C^n$ and let $\{\epsilon_1,\cdots,\epsilon_n\}$ be the frame of $E$, where $\epsilon_j=\PD{x_j}$, with $x_j+i\cdot y_j$ the $j$-th coordinate of $\C^n$. By deforming the frame near the boundaries, we may assume that $\epsilon_j|_{\bS\times[0,\epsilon)}=\pi^*(\epsilon_j|_{\bS})$. Then $\nabla=d+A$ where $A$ is an $End(E)$-valued 1-form on $\Si$, which is defined by
\begin{equation}\label{eq:A}
\nabla\epsilon_i(z)=\sum_i (A)_{ji}(z)\cdot\epsilon_j(z)
\end{equation}
for the standard hermitian metric $h$. Define a connection on $\OL\Si\times\C^n\to\OL\Si$ by $\OL\nabla:=d+\OL{A}$, i.e.,
\begin{equation}
(\OL{A})_{ij}(z):=\OL{(A)_{ij}(\sigma(z))}
\end{equation}
where $\sigma:\Sigma_\C\rightarrow\Sigma_\C$ is the involution map.
\begin{proposition-definition}\label{pro-def:double}
We define a connection $\nabla_\C$ on $E_\C$ which restrict to
$\nabla_\C|_\Si\equiv\nabla$ and $\nabla_\C|_{\OL\Si}\equiv\OL\nabla$. Namely, $A$ is compatible to $\OL{A}$ on the tubular neighborhood $N$ of $\bS$.
\end{proposition-definition}
\begin{proof}
On $\bS$, we fix a starting point $z\in R_i\subset\bS$ and parameterize $R_i$ by $\gamma:[0,1]\to R_i$ with $\gamma(0)=\gamma(1)=z$. Image of Lagrangian subbundle $L_{\gamma(t)}$ under the map $\Phi$ can be written as $u(t)\cdot\R^n\subset\C^n$. In fact, we choose $u(t)$ as follows:
Consider $u(0)$ with its column vectors $(e_1(0),\cdots,e_n(0))$. Using $\nabla$ on $(E,L)$, denote parallel transport of $e_j(0)$ at $\gamma(t)$ by $e_j(t)$. We have $\nabla e_j(t)=0$ on $R_i$.  And $u(t)$ with its column vectors $(e_1(t),\cdots,e_n(t))$, is a unitary matrix with $\Phi(L_{\gamma(t)})=u(t)\cdot\R^n$. Denote entries of $u(t)$ as $(e_{ij}(t))$.
Since $u(t)u(t)^*=I$,
\begin{equation}
\epsilon_i(t)=\OL{e_{i1}(t)}e_1(t)+\cdots+\OL{e_{in}(t)}e_n(t).
\end{equation}
\begin{align}
\nabla\epsilon_i(t)&=\nabla(\sum_{j=1}^n\OL{e_{ij}(t)}e_j(t))\\
    &=\sum_j d\OL{e_{ij}}\otimes e_j+\OL{e_{ij}}\nabla e_j\\
    &=\sum_j d\OL{e_{ij}}\otimes e_j
\end{align}
From the equation (\ref{eq:A}) on $\Si$, if we represent $A$ with respect to the frame $\{\epsilon_1,\cdots,\epsilon_n\}$, the i-th column of $A(t)$ is equal to $\nabla\epsilon_i(t)$ and it is a linear combination of $e_j$'s with coefficients $(d\OL{e_{i1}},\cdots,d\OL{e_{in}})$ by last equation. Since this coefficient vector is i-th row of $d\OL{u(t)}$, we have
$$A(t)=u(t)\cdot\frac{\partial{\OL{u(t)}^T}}{\partial{t}}dt, \; \textrm{on} \; \gamma(t)\in\bS.$$ Note that since $\nabla$ is a product form near the boundary, $A(t,r)=A(t)$ on $N$, where $r$ is the normal coordinate, i.e., $(t,r)\in\bS\times[0,\epsilon)=N$.
Recall that, in the construction of $E_\C$, the transition map was given by the inverse of $B_i$. Note that $B_i(\gamma(t))=u(t)\cdot u(t)^T$.
Hence, under the transition map, the connection 1-form $A$ is transformed to
\begin{align}
\tilde{A(t,r)}&:=B(t)^{-1}\cdot dB(t)+B(t)^{-1}\cdot A(t)\cdot B(t)\\
    &=\OL{u(t)}\cdot\OL{u(t)}^T d(u(t)\cdot u(t)^T)+\OL{u(t)}\cdot\OL{u(t)}^T\cdot A(t)\cdot u(t)\cdot u(t)^T\\
    &=[\OL{u(t)}\cdot\OL{u(t)}^T\frac{\partial{u(t)}}{\partial{t}}u(t)^T
        +\OL{u(t)}\cdot\frac{\partial{u(t)^T}}{\partial{t}}
        +\OL{u(t)}\frac{\partial{\OL{u(t)}^T}}{\partial{t}}u(t)u(t)^T]dt\\
    &=[\OL{u(t)}\PD{t}\{\OL{u(t)}^T u(t)\}u(t)^T
        +\OL{u(t)}\cdot\frac{\partial{u(t)^T}}{\partial{t}}]dt\\
    &=\OL{u(t)}\cdot\frac{\partial{u(t)^T}}{\partial{t}}dt\\
    &=\OL{A(t,r)}
\end{align}
Hence the connection $\nabla$ and $\OL\nabla$ can be pasted near $\bS$.
\end{proof}

Now we start the second proof of Theorem \ref{thm:main}.
\begin{proof}
We first consider the case that orthogonal connection $\nabla$ is of product form near the
boundary. Later,  we prove that $\mu_{CW}$ is independent of the choice of orthogonal connection.

Note that $F_{\OL\nabla}\circ\sigma=d\OL{A}+\OL{A}\wedge\OL{A}=\OL{dA+A\wedge A}=\OL{F_\nabla}$.
Since $\nabla$ is unitary connection, $tr(F_\nabla)$ is purely imaginary.
Hence
\begin{equation}
tr(F_\nabla)+tr(F_{\OL\nabla}\circ\sigma)=tr(F_\nabla)+tr(\OL{F_\nabla})=0.
\end{equation}
Recall that the holomorphic structure of $\OL\Si$ is given by antiholomorphic structure of $\Si$, hence the orientation of $\OL\Si$ is reversed to the one of $\Si$.
\begin{equation}
\int_\Si tr(F_\nabla)=-\int_\Si tr(F_{\OL\nabla}\circ\sigma)=\int_{\OL\Si} tr(F_{\OL\nabla}).
\end{equation}
Since the first Chern number of doubling $E_\C$ gives the Maslov index of $(E,L)$,
\begin{equation}
\frac{\I}{\pi} \int_{\Si}tr(F_{\nabla}) =  \frac{\I}{2\pi}\int_{\Si_\C}tr(F_{\nabla_\C})=c_1(E_\C)([\Si_\C]) = \mu(E,L).
\end{equation}
Therefore we can conclude that
$$\mu_{CW}(E,L)=\frac{\I}{\pi}\int_\Si{tr(F_\nabla)} = \mu(E,L),$$ for a connection which is a product form near the boundary.
This proves the theorem \ref{thm:main} together with the following lemma proposition which
claims that $\mu_{CW}(E,L)$ is the same for any $L$-orthogonal unitary connection.
\end{proof}

\section{Properties of Chern-Weil Maslov index}
In this section, we prove several properties of $\mu_{CW}$.
We  prove that $\mu_{CW}$ is independent  of the choices of orthogonal connection and of
compatible   complex structures. Although this follows from the equivalence which is proved in the previous section, but the proofs given here will naturally extend to the case of orbifolds.

We also give a couple of examples to demonstrate that  it is important to have unitary condition
in the definition of orthogonal connection at the end of the section.
\subsection{Independence of $\mu_{CW}$}
\begin{proposition}\label{lem:stoke}
$\int_\Si tr(F_\nabla)$ is independent of the choice of an $L$-orthogonal unitary connection $\nabla$.
\end{proposition}
\begin{proof}
Let $\nabla^1$ and $\nabla^2$ are $L$-orthogonal unitary connections. Then $\nabla^1-\nabla^2=A$ for some $A\in\Omega^1(\Si)\otimes End(E)$, and we have
 $$tr(F_{\nabla^1})-tr(F_{\nabla^2})=d(tr(A)).$$
To prove the lemma, it is enough to show that $\int_\Si d(tr(A))=\int_{\bS} tr(A)=0$.
After fixing a compatible complex structure $J$ and a trivialization $\Phi:E\to\Si\times\C^n$, let $\{e_1,\cdots,e_n\}$ be an real orthonormal frame of $L$. Then
$$(\nabla^1-\nabla^2)_{\PD{t}}e_i=\sum_j A_{ji}(\PD{t})e_j$$
where $t$ is a local coordinate of $\bS$.
Note that $A_{ji}(\PD{t})$s are real-valued functions over $\bS$ since $\nabla^k (k=1,2)$ preserves $L$. Hence $\int_{\bS} tr(A)$ is real-valued and it vanishes since $tr(F_{\nabla^k})$ are imaginary.

\end{proof}

One can define the notion of isomorphism between two bundle pairs over $\Sigma$, and
it is easy to show that isomorphic bundle pairs have the same Maslov index.
If the bundle pair is defined from a smooth map $u:(\Sigma,\partial \Sigma) \to (M,L)$ to
symplectic manifolds via pulling back tangent bundles, then homotopic maps
define isomorphic bundle pairs, hence has the same Maslov index.

The following corollary also follows from the equivalence $\mu = \mu_{CW}$, but
we give a direct proof.
\begin{corollary}
The Maslov index $\mu_{CW}(E,L)$ of a bundle pair $(E,L)$ does not
depend on a compatible   complex structure $J$.
\end{corollary}
\begin{proof}
Suppose we have a symplectic vector bundle $E \to \Sigma$ and
Lagrangian subbundle $L \to \partial \Sigma$. Let $J_0$, $J_1$ be two
compatible   complex structures of $E$.
It is well-known that the set of compatible complex structures are
path connected, and take $J_t$ connecting $J_0$ and $J_1$.

Now, consider a symplectic bundle $\WT{E} = E \times I \to \Sigma \times I$ and
Lagrangian subbundle $L \times I  \to \partial \Sigma \times I$.

We can choose an $(L\times I)$-orthogonal unitary connection $\WT{\nabla}$ on the complex vector bundle $\WT{E}$ with complex structure $\{J_t\}$.
Here, $(L\times I)$-orthogonal unitary connection means a unitary connection which preserves $(L\times I)$-subbundle along $\bS\times I$.
Then $\WT{\nabla}$ restricts to an $(L \times \{t\})$-orthogonal unitary connection $\WT{\nabla}_t$ on
$$\WT{E}_t = E \times \{t\} \to \Sigma \times \{t\}.$$


It is enough to show that
$$\int_{\Sigma} tr(F_{\WT{\nabla}_0}) = \int_{\Sigma} tr(F_{\WT{\nabla}_1}).$$
As we have
\begin{equation}
0=\int_{\Si\times I}tr(dF_{\WT{\nabla}})=\int_{\Si\times\{1\}}tr(F_{\WT{\nabla}_0})-\int_{\Si\times\{0\}}tr(F_{\WT{\nabla}_1})+\int_{\partial{\Si}\times I}tr(F_{\WT{\nabla}})
\end{equation}
it is enough to show that imaginary part of $\int_{\partial{\Si}\times I}tr(F_{\WT{\nabla}})$ vanishes.

Note that over $\partial{\Si}\times I$, the bundle
$E \times I$ with $\{J_t\}$ is isomorphic to the complexification of real bundle
$L \times I$. Using a similar argument in the Proposition \ref{lem:stoke}, it is easy to show that $tr(F_{\WT{\nabla}})|_{\bS\times I}$ is indeed real-valued.
This proves the corollary.
\end{proof}

\subsection{On the unitary condition.}
Note that the property of the $L$-orthogonal unitary $\nabla$ which preserving the hermitian product along the $\bS$ is important in the above proof. The hermitian property guarantee that the induced connection $\WT\nabla$ on the determinant line bundle $det(E)|_{\bS}$ can be identified with one of $U(1)$-principal line bundle over $\bS$. If we drop the such condition from connection and choose a connection which preserves only the Lagrangian subbundle data over $\bS$, the Chern-Weil definition of Maslov index fails. It is because the curvature integral captures not only the rotations of horizontal sections, but also the change of norm of them. See the following example.

\begin{example}
Consider a bundle pair $E := D^2 \times\C\to D^2$ with a trivial Lagrangian sunbundle $L:=\partial D^2 \times\R$. Define a connection $\nabla:=d+r d\theta$ with respect to the standard complex frame $\{\epsilon\}$ of $E$. Then $\nabla$ preserves the Lagrangian structure, since
\begin{align}
\nabla_\PD{\theta} f\epsilon=0 \Leftrightarrow f(\theta)=f(\theta_0)e^{\theta_0}e^{-\theta}.
\end{align}
Note that
\begin{align}
\int_{D^2} F_\nabla &= \int_{D^2}dr\wedge d\theta\\
&=2\pi
\end{align}
Note that the $2\pi$ measures the ratio of the change of the norm of parallel sections along $D^2$.
\end{example}

\section{The case with boundary condition on transversely intersecting Lagrangian submanifolds}
Recall that to define Fukaya category, $J$-holomorphic maps from holomorphic polygons with boundary on several Lagrangian submanifolds (which intersects transversely) are used. There exist a Maslov index attached to such a map, which determines the virtual dimension of the moduli spaces of such maps.
In general, one can consider maps from Riemann surfaces with boundary with  boundary condition on transversely intersecting Lagrangian submanifolds.
In this section, we give a Chern-Weil definition of the Maslov index of a bundle pair arising from such maps, and find a relation with the virtual dimension of the related moduli spaces.

For simplicity of exposition, we assume that Riemann surface $\Sigma$ has a boundary $\partial \Sigma$ which is connected. (In the general case, the same thing holds by  taking the Maslov index in the sense of Definition \ref{def:mas}.)
We consider marked points (or punctures) $v_0,\cdots,v_k \in \partial \Sigma$ placed in a cyclic order for the induced orientation of $\partial \Sigma$. Holomorphic polygons are genus 0 cases.

\subsection{Orthogonal connection on a bundle pair in transversal case}
Let $\Si$  be a Riemann surface with boundary with vertices labeled as $\{v_0,\cdots, v_k\}$ and
with $k+1$ edges  labeled as $\{l_0, l_1, \cdots, l_k\}$
such that  $v_i = l_i \cap l_{i+1}$ for $i =0,\cdots, k$
modulo $k+1$.
For each $i$, we fix a small closed
neighborhood $U_i$ of $v_i$ and a conformal isomorphism
$$U_i\backslash\{v_i\}\to(-\infty,0]\times[0,1].$$
Let $(M,\omega)$ be a symplectic manifold with a compatible  almost complex structure $J$
and $L_0,\cdots,L_k$ be Lagrangian submanifolds intersecting transversely in $M$. Suppose $p_i \in L_i \cap L_{i+1}$ for
$i=0,\cdots, k$ modulo $k+1$.

Let $u:\Si \to M$ be a $J$-holomorphic map with boundary condition $u(l_j) \subset L_j$ and
asymptotic condition $ lim_{z \to v_i} u(z) =p_i$.
By pulling back via $u$ the tangent bundles, we obtain the following notion of bundle pair.
\begin{definition}
We denote by a {\em bundle pair $(E, \textbf{L}) \to \Si$}, a symplectic  vector bundle $E$ over $\Si$
with Lagrangian subbundles $\textbf{L}:=\{L_0, L_1,\cdots,L_k\}$
 over the edges $\{l_0, l_1, \cdots, l_k\}$ of $\Si$ with  $J$ a compatible   complex structure of $E$ . At $v_i$, Lagrangian subbundles $L_i|_{v_i}$ and $L_{i+1}|_{v_i}$ are assumed to  intersect transversely.
\end{definition}

We give a definition of \textbf{L}-orthogonal unitary connection for the bundle pair $(E,\textbf{L})$ over $\Si$.
\begin{definition}
Let $(E,\textbf{L})$ be a bundle pair over $\Si$ as above with $J$.
A unitary connection $\nabla$ on $E$ is called \textbf{L}-orthogonal unitary connection of $(E,\textbf{L})$ if
the connection  $\nabla_{l_i}$ on $E|_{l_i}$, which is obtained by restriction, is $L_i$-orthogonal
on $l_i$, for each $i=0,1,\cdots,k$.
\end{definition}

\begin{lemma}
Given a bundle pair $(E,\textbf{L}) $ over $\Si$, an orthogonal connection exists.
\end{lemma}
\begin{proof}
It is enough to show that such connection exists in a neighborhood of $v_i$. Since then, one can obtain the global one via partition of unities.

For convenience, we identify a neighborhood of $v_i \in \Si$
with
\begin{equation}
Z^{o}:=\{(x,y)\in \R^2|x^2+y^2 < 1, x\geq0, y\geq0\}
\end{equation}
 where $v_i$ corresponds to 0.

Recall from \cite{MS} that we can take unitary trivialization of $E$ over $Z^{o}$,
$\Phi:E \to Z^{o} \times \C^n$. Here $\C^n$ is equipped with the standard complex structure $J_0$
and the standard symplectic form $\omega_0$, and $\Phi^*\omega_0 = \omega$ and
$\Phi^* J_0 = J$.
On the real (resp. imaginary)  axis $R \subset Z^{o}$ (resp. $I \subset Z^{o}$), we have Lagrangian subbundle $L_{i+1}$ (resp. $L_i$) of $ E$.
By modifying the trivialization $\Phi$ ( by multiplying elements of $U(n)$, in the neighborhood of $R$ and $I$),  we may assume that
the image of $L_i$ and $L_{i+1}$ under $\Phi$ is constant along $R$ and $I$ in $\C^n$.

Now choose a trivial connection  on $Z^{o} \times \C^n$ and pull back via $\Phi$ to obtain an orthogonal
connection of $E$ on $Z^{o}$.
\end{proof}

Now, we associate Chern-Weil Maslov index to the above bundle pair as before.
\begin{definition}
Let $\nabla$ be an \textbf{L}-orthogonal unitary connection of $(E,\textbf{L})$. The Maslov index of the bundle pair $(E,\textbf{L})$ is defined by
\begin{equation}
\mu_{CW}(E,\textbf{L}):=\frac{\sqrt{-1}}{\pi}\int_\Si{tr(F_\nabla)}
\end{equation}
\end{definition}
As in the previous case, we have
\begin{proposition}
$\mu_{CW}(E,\textbf{L})$ is independent of the choice of \textbf{L}-orthogonal unitary connection $\nabla$. It is also independent of the choice of an   complex structure.
\end{proposition}
\begin{proof}
The proof is similar to that of Proposition \ref{lem:stoke}.
\end{proof}

Note that we can choose  a compatible   complex structure $J$ satisfying
$$J\cdot L_i|_{v_i} = L_{i+1}|_{v_i}$$ at the marked point $v_i=l_i\cap l_{i+1}$ for each $i=0,1,\cdots,k$. We will use such a $J$ in the following discussions.

We recall the usual topological Maslov index associated to the bundle pair $(E, \bf{L})$.
First, given two Lagrangian subspaces $L_0$ and $L_1$ which intersects transversely in $V$, there exist a path from $L_0$ to $L_1$ that moves in the positive definite direction (which is unique up to
fixed end points). If $L_1 = J \cdot L_0$, then such a path can be taken to be $t \mapsto e^{\pi J t /2} L_0$.  For example, a loop in Lagrangian grassmanian obtained by joining positive definite paths from $L_0$ to $L_1$ and  from $L_1$ to $L_0$  has Maslov index $n = dim(L_0)$.

The topological Maslov index of the bundle pair $(E, \bf{L})$ can be defined by
first taking a trivialization of $E$ and taking a loop of Lagrangian subspaces along the boundary,
by gluing the Lagrangian subbundle data of the edges  at each marked point $v_i$ via positive definite direction path from $L_i$ to $L_{i+1}$. Denote this path by $L_{loop}$.
The winding number of  $L_{loop}$ defines Maslov index $\mu_{top}(E,\bf{L})$.

We also recall how the Fredholm index arises in this setting.
For a fixed $p>2$, consider a Banach manifold $\mathcal{P}$ of $W^{1,p}$ maps $\Si\to M$ with boundary condition $u(l_j) \subset L_j$ and
asymptotic condition $ lim_{z \to v_i} u(z) =p_i$. Then the moduli space of $J$-holomorphic maps from $\Si$ to $M$ can be identified with a zero set of a smooth section $\overline\partial_J:\mathcal{P}\to\mathcal{E}$ for a Banach bundle $\mathcal{E}$. More precisely, the fiber of $\mathcal{E}$ at $u$ is the space $L^p(\mathcal{A}^{0,1}\bigotimes u^*TM)$, and the section $\overline\partial_J$ is the antiholomorphic part of $du$ with respect to $J$.
If we linearize the $\overline\partial_J$ at $u\in\overline\partial_J^{-1}(0)$ and composite it with the projection map from $T_u\mathcal{E}$ to the fiber $\mathcal{E}_u$, we have a Fredholm operator $D_u\overline\partial_J:W^{1,p}(\Si,u^*TM)\to L^p(\mathcal{A}^{0,1}(\Si)\bigotimes u^*TM)$. The virtual dimension of the component of $\overline\partial_J^{-1}(0)$ containing $u$ is defined by the Fredholm index of $D_u\overline\partial_J$.
We denote the linearized Fredholm operator as $\overline{\partial}_{E,\textbf{L}}$.

Now, we plan to compare the Fredholm index of
$\overline{\partial}_{E,\textbf{L}}$ with $\mu_{CW}(E,\textbf{L})$,
$\mu_{top}(E,\bf{L})$.

\begin{proposition} We have
\begin{eqnarray}
Ind(\bar\partial_{E,\textbf{L}}) + (k+1)\frac{n}{2} &=&\mu_{CW}(E,\textbf{L})+ n \chi(\Sigma)  \label {eq:ind0} \\
Ind(\bar\partial_{E,\textbf{L}}) + (k+1)n &=& \mu_{top}(E,\textbf{L}) + n \chi(\Sigma). \label{eq:ind}
\end{eqnarray}
\end{proposition}
\begin{proof}
From \cite{KL} Theorem 3.4.2, we have
$$Ind(\bar\partial_{E,L}) = \mu(E,L) + n \chi(\Sigma),$$
where $\chi(\Sigma)$ is the Euler characteristic of $\Si$.

The case of boundary condition on transversally intersecting  Lagrangian submanifolds (with $k+1$ marked points),  can be seen from the gluing principle of indices: At each marked point $v_i$, consider
$$Z_- = \{z \in \C| |z| \leq 1 \} \cup \{ z \in \C| Re\, z \geq 0, | Im \,z| \leq 1 \}.$$
and consider $Z_- \times E|_{v_i}$ and Lagrangian boundary condition for $Im \, z = -1$ is $L_{i+1}|_{v_i}$ and for $Im \, z = +1$ is $L_{i}|_{v_i}$ and Lagrangian boundary condition on the arc (left side of the unit circle) is given by a path of positive definite direction from $L_{i}|_{v_i}$ to $L_{i+1}|_{v_i}$.
Denote by $\lambda$ the above Lagrangian bundle data.
Recall from \cite{FOOO} that the index of $\bar\partial_{\lambda,Z_-}$ of
weighted Cauchy-Riemann operator is $n$.
By gluing $Z_-$ at each marked point, we obtain the equation  \eqref{eq:ind}.
(We refer readers to \cite{FOOO} for more details on this argument).

To prove the first identity, we find a relation between
$\mu_{top}(E,\textbf{L})$ and $\mu_{CW}(E,\textbf{L})$ by studying
the index of a basic piece. Instead of $Z_-$, we consider the
following domain $Z$ to compute $\mu_{CW}$ in an easier way.
\begin{equation}\label{Z}
Z:=\{(x,y)\in \R^2|x^2+y^2 \leq 1, x\geq0, y\geq0\}
\end{equation}

We consider  the following bundle  pair on $Z$. Consider $Z \times \C^n \to Z$ equipped
with the standard symplectic structure, and we describe Lagrangian subbundle on $\partial Z$.
Note that the boundary $\partial Z$ consists of three parts, $R, I, A$ which are real axis, imaginary axis and
arc respectively.

By identifying $E_{v_i} \cong \C^n$ (so that $\omega, J$ becomes $\omega_0, J_0$),  denote by $\Lambda$ (resp. $\Lambda'$)
the Lagrangian subspace of $\C^n$ corresponding to $L_{i}|_{v_i}$ (resp. $L_{i+1}|_{v_i}$).
Note that $J_0 \Lambda = \Lambda '$.

We define the Lagrangian subbundle over $R$  to be the  constant $R \times \Lambda \subset R \times \C^n$,
over $I$ to be the constant $I \times \Lambda' \subset I \times \C^n$ and over $A$ in the counterclockwise
direction, to be a positive definite direction path in Lagrangian grassmanian from $\Lambda$ to $\Lambda'$.  We assume that the path on $A$ is constant near the axis $I$ or $R$.
We may denote this Lagrangian subbundle on $\partial Z$ by $\Lambda$.

Now, we compute $\mu_{CW}(Z \times \C^n, \Lambda)$. Now we can take
an $\Lambda$-orthogonal connection as follows.  We choose a map
$\gamma:[0,1]\rightarrow U(n)$ whose column vectors form a unitary
frame of $\Lambda$ on the arc $A \subset \partial Z$ such that
$$\gamma(1)= e^{\pi i/2}\cdot\gamma(0),$$
and constant in $U(n)$ near end points 0 and 1. We may take a
connection $\nabla$ on the bundle satisfying $\nabla
(\gamma\cdot\epsilon_j) \equiv0$ (for $j=1,\cdots,n$) near the arc,
and $\nabla\equiv d$ near the real axis and imaginary axis. This
defines a $\Lambda$-orthogonal connection and we define
$$\mu_{CW}(Z \times \C^n, \Lambda):=\frac{\I}{\pi}\int_Z{tr(F_\nabla)}.$$

\begin{lemma}
$\mu_{CW}(Z \times \C^n, \Lambda)$ is equal to $\frac{n}{2}$ which is the half of topological Maslov index of loop $\gamma*( e^{\pi i/2}\cdot\gamma)$ in $\mathcal{L}(n)$, where $\gamma*(e^{\pi i/2}\cdot\gamma)$ is a smooth function from [0,2] to $U(n)$ defined by
\begin{align}\label{quad}
\gamma*( e^{\pi i/2} \cdot\gamma)(t):=
\begin{cases}
\gamma(t) & \text{if $t\in[0,1]$}\\
e^{\pi i/2}\cdot\gamma(t-1) & \text{if $t\in[1,2]$}
\end{cases}
\end{align}
\end{lemma}

 \begin{proof}
 Consider a (trivial) complex vector bundle on $D^2$. Consider a Lagrangian subbundle $L_\Gamma$ over $\partial D^2$ by concatenating four paths $\Gamma:=\gamma*(e^{\pi i/2}\cdot\gamma)*(e^{\pi i}\cdot\gamma)*(e^{3\pi i/2}\cdot\gamma)$  as in (\ref{quad}).
 Note that, since $e^{3\pi i/2}\cdot\gamma(1)=\gamma(0)$, this path $\Gamma$ is a loop in $U(n)$  and hence  loop in $\mathcal{L}(n)$. Thus if $\nabla$ is an orthogonal connection, $\frac{\I}{\pi}\int_D{tr(F_\nabla)}$ gives the topological Maslov index of $(D^2\times\C^n,L_\Gamma)$ from the Theorem \ref{thm:main}.

Now, we construct an orthogonal connection $\nabla$ on $D^2$ from that of $Z$.
 Note that the $\nabla$ constructed on $Z$ induces a connection on $D^2$ by pullback of the map $m_k:D^2\rightarrow D^2$, $z\mapsto e^{\frac{k \pi i}{2}}\cdot z$ for $k=0,1,2,3$. Then, since $\nabla\equiv d$ near the real axis and imaginary axis, the pullback connection on each $e^{\frac{k\pi i}{2}}\cdot Z$ $(k=0,1,2,3)$ can be glued to give a connection on $D$. It is easy to see that this connection is indeed orthogonal connection for the pair  $(D^2\times\C^n,L_\Gamma)$.
 Note that the curvature integral on each quadrant of $D^2$ are the same:
$$\int_Z tr(F_\nabla)=\int_{m_k^{-1}(Z)} tr(m_k^*F_{\nabla})=\int_{m_k^{-1}(Z)} tr(F_{m_k^*\nabla}).$$
Hence we have
 \begin{align}
 \frac{\I}{\pi}\int_Z tr(F_\nabla)&=\frac{1}{4} \frac{\I}{\pi} \int_{D^2} tr(F_\nabla) \\
 &=\frac{1}{4}\text{(topological Maslov index of $\Gamma$ in $\mathcal{L}(n)$)} \\
 &=\frac{1}{2}\text{(topological Maslov index of $\gamma*(i\cdot\gamma)$ in $\mathcal{L}(n)$)}\\
 &= \frac{n}{2}
 \end{align}
The last identity follows from the fact that since each $\gamma$ is
chosen to be positive definite, $\gamma*(e^{\pi i/2}\cdot\gamma)$ is
a positive definite loop whose Maslov index is equal to $n$, which
is the dimension of Lagrangian subspace.
\end{proof}

Now, as before, we attach bundle pair over $Z$ to that over $\Sigma$ at each marked point.
After attaching $(k+1)$ copies of bundle data on $Z$, the resulting Lagrangian subbundle along
the boundary  is obtained by connecting $L_i$ and $L_{i+1}$  by positive definite direction paths and
it becomes $L_{loop}$ which defines $\mu_{top}$.
Thus we have the following identity.
\begin{equation}
\displaystyle\mu_{top}(E,\textbf{L})=\mu_{CW}(E,\textbf{L})+\sum_{i=0}^k\mu_{CW}(Z\times
{\C}^n,\Lambda) = \mu_{CW}(E,\textbf{L}) + \frac{(k+1)n}{2}
\end{equation}
Combining with the equation \eqref{eq:ind}, we obtain the equation \eqref{eq:ind0}.
This finishes the proof of the proposition.
\end{proof}

When $k=0$, and $\Si$ is a bi-gon,
$Ind(\bar\partial_{E,\textbf{L}})$  which equals $\mu_{top}(E,\textbf{L}) -n$, is called
 the Maslov-Viterbo index.
Thus we obtain the following corollary.
\begin{corollary}
\begin{equation}
\textrm{Maslov-Viterbo
index}=Ind(\bar\partial_{E,\textbf{L}})=\mu_{CW}(E,\{L_0,L_1\})
\end{equation}
\end{corollary}

%
%
%

\section{Orbifold Maslov Index}
In this section, we extend the definition of Maslov index to the case of orbifolds.
Namely, consider an orbi-bundle over a
bordered orbifold Riemann surface with interior singularities, and a Lagrangian subbundle along the boundary.  Note that orbi-bundles in these cases are not trivial bundles, and hence topological
definition of Maslov index is not directly extended to these cases.

But Chern-Weil definition extends naturally by requiring the connection to be invariant under local group actions near orbifold singularities.
Using the Chern-Weil definition, we show that there is a well-defined topological definition
of orbifold Maslov index.

At the end of this section, we show a relation of orbifold Maslov index and desingularized Maslov
index introduced by Poddar and the first author \cite{CP}, using the desingularization procedure
introduced by Chen and Ruan \cite{CR}.

\subsection{Orbifold Chern-Weil Maslov index}
We first recall the definition of bordered orbifold Riemann surface
and $J$-holomorphic maps to  almost complex orbifolds. We denote
$\bz=(z_1,\cdots,z_k), \bm=(m_1,\cdots,m_k)$ in the following.
\begin{definition}
Let $\Sigma$ be a bordered Riemann surface with complex structure
$j$. $(\Sigma,\bz, \bm)$ is called a bordered orbifold Riemann
surface with interior singularities if $\bz$ are distinct interior
of $\Sigma$, and if a disc neighborhood of each $z_i$ is uniformized
by a branched covering map $z \mapsto z^{m_i}$.
\end{definition}
Thus the disc neighborhood $U_i$ of $z_i$ is understood as a quotient space of
$D^2$ by the standard rotation action of the local group $\Z/m_i\Z$.
We denote by $\bsg = (\Sigma,\bz, \bm)$ the orbifold bordered Riemann surface.

In our case, we can also consider a smooth Riemann surface $\WT{\Sigma}$ with a
branch covering map $\pi:\WT{\Sigma} \to \Sigma$ such that
the orbifold $\bsg$ is obtained as the quotient of $\WT{\Sigma}$ by the
action of deck transformation group $G$  of $\pi$ ( i.e. $\bsg$ is good orbifold).

Such $\WT{\Sigma}$ can be obtained as follows.
Consider two copies of $\bsg$ labelled as $\bsg_1, \bsg_2$, and glue $\bsg_1$ with $\OL{\bsg_2}$(opposite orientation) to obtain $\bsg_{double}$, which becomes a good orbifold.
Hence it has a  smooth Riemann surface $\WT{\Sigma}_{double}$ with branch covering  $\pi :\WT{\Sigma}_{double} \to \bsg_{double}$. By considering only $\WT{\Sigma}:=\pi^{-1}(\bsg_1)$,
we obtain the desired Riemann surface with boundary $\WT{\Sigma}$.

Now, consider an orbifold vector bundle $E \to \bsg$. (see for example \cite{CR}).
On the neighborhood $U_i$ of $z_i$ above,  orbifold vector bundle $E|_{U_i} \to U_i$ has a uniformizing
chart $ D^2 \times \R^n \to D^2$ together with $\Z/m_i\Z$-action compatible with the orbifold
structure of $U_i$.
This may be understood as a genuine vector bundle $\WT{E} \to \WT{\Sigma}$ with an action of
deck transformation group $G$, which is compatible with that of $\WT{\Sigma}$.

Also recall that a connection $\nabla$ on orbifold vector bundle $E \to \bsg$ is defined to be
invariant under local group action.

We define a bundle pair over $(\bsg,\partial \bsg)$.
\begin{definition}
We denote by a {\em bundle pair $(E, L) \to (\bsg,\partial \bsg)$}, an orbifold symplectic  vector bundle $E$ over $\bsg$ and a Lagrangian subbundle   $L$ over $\partial \Si$.
\end{definition}
We choose a compatible complex structure $J$ of $E$.
The bundle data in the orbifold case arises is obtained by a good map from $(\Sigma,\bz, \bm)$ to a symplectic orbifold with Lagragnian boundary condition.
The notion of a good map was introduced by Chen and Ruan (which we refer readers to \cite{CR}), which enables one to pull-back bundles. Given a $J$-holomorphic map which is a good map, we
obtain a bundle pair by pull-back  tangent bundles.

We define $L$-orthogonal connection of a bundle pair as follows
\begin{definition}
Let $(E,L)$ be a bundle pair over $(\bsg,\partial \bsg)$. A unitary connection
$\nabla$ on $E$ is called orthogonal connection
if the parallel transport  along $\bS$ via $\nabla$ preserves Lagrangian subbundle $L \to \bS$.
\end{definition}

Now, we give a definition of the Maslov index $\mu_{CW}$ for $(E,L)\to (\Sigma,\bz, \bm)$ in terms of curvature integral:
\begin{definition}\label{def:mu}
Let $\nabla$ be an orthogonal connection of a bundle pair $(E,L)\to (\bsg,\partial \bsg)$.
We define the Maslov index of the bundle pair $(E,L)$ as
$$\mu_{CW}(E,L)=\frac{i}{\pi}\int_\Si{tr(F_\nabla)}$$
where $F_\nabla\in\Omega^2(\Si,End(E))$ is the curvature induced by $\nabla$.
\end{definition}

As in the previous cases, we have
\begin{proposition}
$\mu_{CW}(E,L)$  in Definition \ref{def:mu}  is independent of the choice of $L$-orthogonal  connection $\nabla$. It is also independent of the choice of a complex structure $J$.
\end{proposition}
\begin{proof}
The proof is similar to that of Proposition \ref{lem:stoke}, using Stoke's theorem in the orbifold setting.
\end{proof}

\subsection{Topological definition of orbifold Maslov index}
One  possible approach to define Maslov index topologically in the orbifold case is as follows.

\begin{definition}
Consider a bundle pair $(E,L)\to (\bsg,\partial \bsg)$.
Take branch covering
 $\pi:\WT{\Si} \to \bsg$ by a smooth Riemann surface $\WT{\Si}$ with boundary, and
 consider pull-back bundles $(\pi^*E , \pi^*L)$ which becomes a smooth bundle pair
 on $(\WT{\Si},\partial \WT{\Si})$.

We define $$\mu_{\pi}(E,L) = \frac{1}{|G|} \mu (\pi^*E, \pi^*L)$$
 where $|G|$ is the degree of the branch covering map $br$.
\end{definition}
A priori, it is not clear (at least for the authors)  if  $\mu_{\pi}(E,L)$ is independent of the choice of the branch covering map $\pi$. But we use Chern-Weil definition of Maslov index to prove that
such a topological index is independent of the choice of a branch covering map, and prove that it is the same as Chern-Weil Maslov index. This should be useful in actual computations of Maslov indices.

\begin{proposition}
We have
$$\mu_{\pi}(E,L) = \mu_{CW}(E,L).$$
In particular, $\mu_{\pi}(E,L)$ is independent of the choice $\pi$ of the branch covering map.
\end{proposition}
\begin{proof}
Let $\nabla$ be an $L$-orthogonal connection on an orbifold vector bundle pair $(E,L) \to (\bsg,\partial \bsg)$. Let $\pi^*\nabla$ be a pull-back connection on $\pi^*E$, which becomes
a $L$-orthogonal connection of bundle pair $(\pi^*E, \pi^*L)$

By theorem \ref{thm:main}, we have $\mu (\pi^*E, \pi^*L) = \mu_{CW}(\pi^*E, \pi^*L)$.
Therefore,  we have
 $$\mu (\pi^*E, \pi^*L) = \mu_{CW}(\pi^*E, \pi^*L)= \int_{\WT{\Sigma}} F_{\pi^*\nabla}
 = |G| \int_{\Sigma} F_{\nabla} = |G| \mu_{CW}(E,L).$$
\end{proof}

In fact, one may observe that the above argument works for branch coverings between two
smooth bundle pairs also.

Consider a branched covering $\phi:\Sigma_1 \to \Sigma_2$ of degree $m$ for
bordered Riemann surfaces $\Sigma_1,\Sigma_2$.
(Here we assume that the branching
locus lies in the interior of $\Sigma_2$)
Then, given a smooth map $u : (\Sigma_2 ,\partial \Sigma_2) \to (M,L)$ for
a symplectic manifold $M$ and Lagrangian submanifold $L$, we obtain by composition
another map $u \circ \phi : (\Sigma_1 ,\partial \Sigma_1) \to (M,L)$.

Define the Maslov index of $u$ to be $\mu(u^*TM,u|_{\partial \Sigma_2}^*TL)$,
and similarly for $u \circ \phi$.
Then, the same argument as in the above proposition proves that we have
$$\mu(u \circ \phi) = m \cdot \mu(u).$$

\subsection{Relation to desingularization}
In the rest of the paper, we recall, what is called the desingularization of orbi-bundle from \cite{CR}, and recall
the desingularized Maslov index from \cite{CP}. Then, we will find a relation between
the desingularized Maslov index and the  Maslov index defined in this paper.

We recall the desingularization of orbi-bundle on an orbifold
Riemann surface by Chen and Ruan (\cite{CR}). Consider $\bsg =
(\Sigma,\bz,\bm)$ as before. Let $E$ be a
complex orbifold bundle of rank $n$ over $\bsg$. Then at each
singular point $z_i$, $E$ determines a representation
$\rho_i:\Z_{m_i} \to Aut(\C^n)$ so that over a disc neighborhood
$D_i$ of $z_i$, $E$ is uniformized by $(D_i \times \C^n,
\Z_{m_i},\pi)$ where the action of $\Z_{m_i}$ on $D_i \times \C^n$
is given by

\begin{equation}
e^{2\pi i/m_i} \cdot (z,w)=\big(e^{2\pi i/m_i} z, \rho_i(e^{2\pi i/m_i})w \big)
\end{equation}
for any $w \in \C^n$. Each representation $\rho_i$ is uniquely determined by a $n$-tuple of integers
$(m_{i,1},\cdots,m_{i,n})$ with $0 \leq m_{i,j} < m_i$, as it is given by matrix

\begin{equation}
\rho_i(e^{2\pi i/m_i})= diag( e^{2\pi i m_{i,1}/m_i},\cdots, e^{2\pi i m_{i,n}/m_i} \big).
\end{equation}

Over the punctured disc $D_i \setminus \{0\}$ at $z_i$, $E$ inherits a specific trivialization from
$(D_i \times \C^n,\Z_{m_i},\pi)$ as follows:
We define a $\Z_{m_i}$-equivariant map $\Psi_i:D \setminus \{0\} \times \C^n \to D \setminus \{0\} \times
\C^n$ by

\begin{equation}
(z,w_1,w_2,\cdots,w_n) \to (z^{m_i}, z^{-m_{i,1}}w_1,\cdots,z^{-m_{i,n}}w_n),
\end{equation}
where $\Z_{m_i}$ acts trivially on the second $D\setminus \{0\} \times \C^n$. Hence $\Psi_i$ induces
a trivialization $\Psi_i:E_{D_i \setminus \{0\}} \to D_i \setminus \{0\} \times \C^n$.
We extend the smooth complex vector bundle $E_{\bsg \setminus \bz}$ over $\bsg \setminus \bz$ to
a smooth complex vector bundle over $\bsg$ by using these trivializations $\Psi_i$ for each $i$. The resulting
complex vector bundle is called the desingularization of $E$ and denoted by $|E|$.
The essential point as observed in \cite{CR} is that the sheaf of holomorphic sections of the desingularized orbi-bundle  and the orbibundle itself are the same.

We recall Chen-Ruan's index formula:
\begin{proposition}\label{prop:chernnuber}
The Chern number of orbi-bundle and that of its de-singularization satisfies
(Proposition 4.2.1 \cite{CR})
$$ c_1(E)([\Sigma]) = c_1(|E|)([\Sigma]) + \sum_{i=1}^k \sum_{j=1}^n \frac{m_{i,j}}{m_i}.$$
\end{proposition}

Now, as the local group action on the fibers of the desingularized
orbi-bundle $|E|$ is trivial, one can think of it as a smooth vector
bundle on $\Sigma$ which is analytically the same as $E$ (In other
words,  there exist a canonically associated vector bundle $|E|$
over the smooth Riemann surface $\Sigma$). Hence, for the bundle
$|E|$, the ordinary index theory can be applied, which provides the
required index theoretic tools for the orbibundle $E$.

Now, we recall a definition of the desingularized Maslov index, which determines the virtual dimension of the moduli space of J-holomorphic orbi-discs from \cite{CP}
Let $X$ be a symplectic orbifold and $N$ be a Lagrangian submanifold (which do not
contain any orbifold singularity).
Let $\bsg$ be an orbi-disc with interior orbifold singularity $(\bz_1,\cdots,\bz_k)$.
Let $u:(\bsg,\partial \bsg) \to (X,N)$ be an orbifold J-holomorphic disc with Lagrangian boundary condition. Then, $E:=u^*TX$ is a complex orbi-bundle over $\bsg$, with
Lagrangian subbundle $L:=u|_{\partial \bsg}^*TN$ at $\partial \bsg$.

\begin{definition}
Let $|E|$ be the desingularized bundle over $\bsg$( or $\Sigma$), which still have the Lagrangian subbundle
  at the boundary from $L$.
The Maslov index of the bundle pair $(|E|,L)$ over $(\Sigma,\partial \Sigma)$ is called
 the {\em desingularized Maslov index} of $(E,L)$, and denoted as $\mu^{de}(E,L)$.
\end{definition}

We find a relation of the desingularized Maslov index of \cite{CP} and the Maslov index in this paper.
\begin{proposition}
We have
$$\mu_{CW}(E,L) = \mu^{de}(E,L) + 2 \sum_{i=1}^k \sum_{j=1}^n \frac{m_{i,j}}{m_i}.$$
\end{proposition}
\begin{proof}
We first consider the double $E_\C$ of the bundle pair $(E,L)$ for
bordered Riemann surface with interior orbifold singularities.
Then we have from Chen-Ruan's formula that
$$c_1(E_\C)([\Sigma_\C]) = deg(E_\C) = c_1(|E_\C|)([\Sigma_\C]) + 2\sum_{i=1}^k \sum_{j=1}^n \frac{m_{i,j}}{m_i}.$$

Note that from \cite{KL}, we have  $\mu^{de}(E,L) =   c_1 ( |E|_\C)([\Sigma_\C])$, and as $|E_\C| = |E|_\C$ holds,
we have  $\mu^{de}(E,L) =  c_1(|E_\C|)([\Sigma_\C])$.

Note that  given an orthogonal connection $\nabla$ on $(E,L)$ over $(\Sigma,\bz,\bm)$,
we can find a connection $\nabla_\C$ on $E_\C$ as in the section 3.2.
From the Chern-Weil definition of Maslov index $\mu_{CW}(E,L)$ over $(\Sigma,\bz,\bm)$,
we find that
$$\mu_{CW}(E,L) = \frac{i}{\pi}\int_{\Sigma}^{orb} tr(F_\nabla)
= \frac{i}{2\pi}\int_{\Sigma_\C}^{orb} tr(F_{\nabla_\C}) = c_1(E_\C)([\Sigma_\C]). $$

Hence, we obtain the proposition.
\end{proof}

\bibliographystyle{amsalpha}

\begin{thebibliography}{}

\bibitem[AG]{AG}
N.L. Alling, N. Greenleaf, {\em Foundations of the theory of Klein
surfaces,} Springer, Berlin, 1971.

\bibitem[Al]{A} G. Alston, {\em Floer cohomology of real Lagrangians in the Fermat quintic
threefold}, arXiv:1010.4073.

\bibitem[Ar]{Ar} V.I. Arnold, {\em On a characteristic class entering into conditions of quantization},
Functional Analysis and its applications 1, 1-14, 1967.

\bibitem[CLM]{CLM} S.Cappell, R.Lee and E.Y.Miller, {\em On the Maslov index,} Comm. Pure Appl. Math. 47, 121-186,
1994.

\bibitem[CP]{CP} C.-H. Cho and M. Poddar,
{\em Holomorphic orbidiscs and Lagrangian Floer cohomology of toric orbifolds}
in preparation.

\bibitem[CR]{CR} W. Chen and Y. Ruan,
{\em A new cohomology theory of orbifold}, Comm. Math. Phys. 248 no.
1, 1-31, 2004.

\bibitem[CR2]{CR2} W. Chen and Y. Ruan,
{\em  Orbifold Gromov Witten theory,} Orbifolds in mathematics and
physics (Madison, WI, 2001), Contemp. Math. 310, 25-85, Amer. Math.
Soc., Providence, RI, 2002.

\bibitem[DHVW]{DHVW} L. Dixon, V. Harvey, C. Vafa and E. Witten,
{\em Strings on orbifolds I., II.,}  Nucl. Phys. B261 (1985), no. 4,
678.

\bibitem[Fu]{Fu} K. Fukaya,
{\em Floer homology and mirror symmetry. II, }
 Minimal surfaces, geometric analysis and symplectic geometry , Adv. Stud. Pure Math., 34, Math. Soc. Japan, 2002. 31–127
\bibitem[FOOO]{FOOO} K. Fukaya, Y.-G. Oh, H. Ohta, K. Ono, {\em Lagrangian Intersection Floer Theory: Anomaly and Obstruction. Parts I and
II.,}, AMS, 2009.

\bibitem[GS]{GS}
V.Guillemin and S.Sternberg, {\em Symplectic geometry, Chapter IV of Geometric Optics,} Math. Surv. and Mon. 14, AMS, 109-202 (1990)

\bibitem[KL]{KL}
S. Katz, C.-C. M. Liu, {\em Enumerative geometry of stable maps with
Lagrangian boundary conditions and multiple covers of the disc,}
Advances in Theoretical and Mathematical Physics, Vol. 5, 1-49,
2002.

\bibitem[M]{M}
V.Maslov, {\em Theory of perturbations and asymptotic methods,}
French translation of Russian original (1965), Gauthier-Villars
(1972)

\bibitem[MS]{MS} D. Mcduff, D. Salamon, {\em Introduction to symplectic
topology,}, Oxford Mathematical Monographs, 1998.

\bibitem[Oh]{Oh} Y-G. Oh, {\em Symplectic topology and Floer Homology},
Book in preparation.

\bibitem[RS]{RS} J. Robbin, D. Salamon,
{\em The Maslov index for paths}, Topology, 32(4):827.844, 1993.

\bibitem[SZ]{SZ} D. Salamon, E. Zender,
{\em Morse theory for periodic solutions of Hamiltonian systems and
the Maslov index}, Communications on Pure and Applied Mathematics,
Vol. 45, 1303-1360, 1992.

\bibitem[V]{V} I. Vaisman,
{\em Symplectic Geometry and Secondary Characteristic Classes,}
Progress in Math., vol.72, Birkhauser, Boston, 1987.
\end{thebibliography}

\end{document}